\documentclass[a4paper,11pt]{amsart}
\usepackage[plainpages=false]{hyperref}
\usepackage{amsfonts,latexsym,rawfonts,amsmath,amssymb,amsthm,mathrsfs}
\usepackage{amsmath,amssymb,amsfonts,latexsym,lscape,rawfonts,comment}
\usepackage[usenames]{color}

\usepackage{graphicx,psfrag}

\usepackage{array,tabularx}

\usepackage{setspace}

\newtheorem{thm}{Theorem}[section]

\newtheorem{prop}[thm]{Proposition}

\newtheorem{prob}[thm]{Problem}

\theoremstyle{remark}

\theoremstyle{definition}
\newtheorem{Def}[thm]{Definition}

\title{A splitting theorem on toric varieties}

\author{Hongnian Huang}
\address{Hongnian Huang, CMLS, \'Ecole Polytechnique, 91128, Palaiseau, France}
\email{hnhuang@gmail.com}

\thanks{The author would like to thank Vestislav Apostolov, Paul Gauduchon and G\'abor Sz\'ekelyhidi for stimulating discussions. His research is financially supported by the Fondation math\'ematique Jacques Hadamard, Paris, France.} 
\begin{document}

\maketitle

\begin{abstract}
Using the short time existence of the Calabi flow, we prove that any extremal K\"ahler metric $\omega_E$ on a product toric variety $X_1 \times X_2$ is a product extremal K\"ahler metric.
\end{abstract}

In \cite{AH}, the authors proposed the following problem: 

\begin{prob}
Let $X_i, i=1,2$ be two K\"ahler manifolds with K\"ahler classes $[\omega_i]$. Suppose $\omega_E$ is an extremal K\"ahler metric in the K\"ahler class $[\omega_1 + \omega_2]$. Can we conclude that $\omega_E$ is a product metric, i.e., $\omega_E = \omega_{E,1} + \omega_{E,2}$ where $\omega_{E,i}$ is an extremal K\"ahler metric in $[\omega_i]$.
\end{prob}
In this short note, we solve the problem in the case of  toric varieties.

\begin{thm}
\label{main}
If in addition $X_i$ are toric varieties, then $\omega_E$ is a product metric.
\end{thm}

Let $X_i, i=1,2$ be toric varieties with K\"ahler classes $[ \omega_i ]$. Let $P_i$ be the corresponding Delzant polytopes. Then the corresponding polytope $P = P_1 \times P_2$. In the symplectic side, we have symplectic potentials $u_i$ satisfying Guillemin boundary conditions of $P_i$. We let $x$ be the variable of $P_1$ and $y$ be variable of $P_2$. Our assumption shows that there exists a symplectic potential $u$ on $P$ and 
$$
u(x,y) = u_1(x) + u_2(y) + f(x,y), \quad f(x,y) \in C^\infty(\bar{P}) 
$$
such that the scalar curvature of $u(x,y)$ is an affine function. Our goal is to show that $f(x,y)$ is separable. Let
$$
f_1(x) = \frac{1}{vol(P_2)} \int_{P_2} f(x,y)~ dy, \quad f_2(y) = \frac{1}{vol(P_1)} \int_{P_1} f(x,y)~ dx.
$$
Then we have

\begin{prop}
$ v(x,y) = u_1(x) + u_2(y) + f_1(x) + f_2(y) $ is a symplectic potential of $P$ satisfying the Guillemin boundary conditions.
\end{prop}

\begin{proof}
It is easy to see that $f_1(x) + f_2(y)$ is a smooth function on $\bar{P}$. Thus we only need to show that $(D^2 v)$ is a positive matrix in order to prove that $v$ is a symplectic potential. To show $(D^2 v) > 0$ is equivalent to show that $(D^2 (u_1(x) + f_1(x))) > 0$ and $(D^2 (u_2(y) + f_2(y))) > 0$. However, $(D^2 (u_1(x) + f_1(x))) > 0$ and $(D^2 (u_2(y) + f_2(y))) > 0$ just follow from the fact that $(D^2 u) > 0$.
\end{proof}

Let $\mathcal{H}$ be the set of all symplectic potentials. We define a subset of $\mathcal{H}$.
\begin{Def}
\begin{eqnarray*}
\mathcal{M} &=& \{ \underline{u}(x,y) \in \mathcal{H} ~ | ~  \underline{u}(x,y)  = u_1(x) +u_2(y) + g_1(x) + g_2(y)~ s.t. \\
& & g_1(x) \in C^\infty(\bar{P}_1), ~ \int_{P_1} f_1(x) ~ dx = \int_{P_1} g_1(x) ~ dx, \\
& & g_2(y) \in C^\infty(\bar{P}_2),  ~ \int_{P_2} f_2(y) ~ dy = \int_{P_2} g_2(y) ~ dy\} .
\end{eqnarray*}
\end{Def}

Then we have
\begin{prop}
\begin{eqnarray}
\label{minimizer}
\int_P (u(x,y) - v(x,y))^2 ~ dxdy = \min_{\underline{u} \in \mathcal{M}} \int_P (u(x,y)- \underline{u}(x,y))^2 ~ dxdy.
\end{eqnarray}
\end{prop}

\begin{proof}
(\ref{minimizer}) is equivalent to show that
$$
\int_P (f(x,y) - f_1(x) - f_2(y))^2 ~ dxdy \leq \int_P (f(x,y)- g_1(x) - g_2(y))^2 ~ dxdy.
$$
Expressing it out, we have
\begin{eqnarray*}
& &\int_P - 2 f(x,y) (f_1(x) + f_2(y)) + f_1^2(x) + f_2^2(y) ~ dxdy \\
&\leq& \int_P -2 f(x,y) (g_1(x) + g_2(y)) + g_1^2(x) + g_2^2(y) ~ dxdy
\end{eqnarray*}
which is equivalent to
$$
0 \leq \int_P (f_1(x) - g_1(x))^2 + (f_2(y) - g_2(y))^2 ~ dxdy.
$$

The equality holds iff $f_1(x) = g_1(x)$ and $f_2(y) = g_2(y)$.
\end{proof}

\begin{proof}[Proof of Theorem (\ref{main}).]
We use the Calabi flow to show that $v$ is an extremal symplectic potential. We need the short time existence of the Calabi flow due to Chen and He \cite{ChenHe} and the modified version of the Calabi flow in \cite{HZ}.
Let $u(t)$ be a sequence of symplectic potentials satisfying the modified Calabi flow equation on $P$ and $u(0) = v$. By Calabi and Chen's result \cite{CC}, we have
$$
\frac{d}{d t} \int_P (u(t) - u)^2 ~ dxdy \leq 0.
$$
Since $u(t) \in \mathcal{M}$, we obtain
$u(t) = v$. It shows that $v$ is a separable extremal symplectic potential on $P$. It is easy to conclude that $f(x,y)$ is separable.
\end{proof}

\end{document}